%% file: RoeckNoteSeptember2016.tex
\documentclass[10pt]{article}
\input epsf.tex
\input{preamble}

\usepackage{amsfonts,amsmath,latexsym,amsthm,amssymb}
\usepackage{dsfont} 
\usepackage{mathtools}
\usepackage{enumerate}
\usepackage{mathrsfs} 
\mathtoolsset{showonlyrefs}
\usepackage{listings} 
\usepackage{amscd}
\def\wt{\widetilde}
\def\pas{\ \d P \mbox{-a.s.}}
\def\<{\left<}
\def\>{\right>}
\def\({\left(}
\def\){\right)}

\def\o{\omega}
\def\O{\Omega}
\def\9{\infty}
\def\R{\mathbb R}
\def\N{\mathbb N}

\def\shb{{\cal B}}
\def\shc{{\cal C}}

\def\shf{{\cal F}}

\def\shm{{\cal M}}

\def\shs{{\cal S}}

\newcommand{\norm}[1]{\left\| #1 \right\|}

\newtheorem{theo}{Theorem}[section]

\newtheorem{Assumption}[theo]{Assumption}

\newtheorem{rem}[theo]{Remark}

\newtheorem{defi}[theo]{Definition}

\newcommand{\beqnar}{\begin{eqnarray*}}
\newcommand{\eeqnar}{\end{eqnarray*}}
\newcommand{\ba}{\begin{array}}
\newcommand{\ea}{\end{array}}

\newcommand{\halb}{\frac{1}{2}}

\begin{document}

\title{Uniqueness for a class of stochastic Fokker-Planck and porous media equations.}

\author{ Michael R\"ockner (1)
and Francesco Russo (2) } 

\date{September 1st 2016}
\maketitle

\thispagestyle{myheadings}
\markright{Stochastic Fokker-Planck equation.}

{\bf Summary.} The purpose of the present note consists of first
showing a uniqueness result for a stochastic Fokker-Planck equation 
under very general assumptions. In particular, the second order
 coefficients may be just measurable and
degenerate. We also provide a proof 
for uniqueness of a stochastic porous media equation in 
a fairly large space.


{\bf2010  AMS-classification}: 35R60; 60H15;  82C31.

\begin{itemize}
\item[(1)] Michael R\"ockner,
Fakult\"at f\"ur Mathematik, 
Universit\"at   Bielefeld, 
\\ D--33615 Bielefeld, Germany.
\item[(2)] Francesco Russo,
ENSTA ParisTech, Universit\'e Paris-Saclay,
Unit\'e de Math\'ematiques appliqu\'ees,
 828, boulevard des Mar\'echaux,
F-91120 Palaiseau, France.

\end{itemize}

\vfill \eject

\section{Introduction}

\setcounter{equation}{0}

We consider  real functions 
$ e^0, \ldots, e^N, \dots: \R \rightarrow \R$ 
 fulfilling Assumption \ref{AMultipl} below.
In particular they are $H^{-1}$-multipliers, 
see Definition \ref{DMultipl}. 

 Let $ T > 0$ and $(\Omega, \shf, P)$,
be a fixed  probability space.
Let $(\shf_t, t \in [0,T])$ be a  filtration
fulfilling the usual conditions and 
we suppose $\shf = \shf_T$.
Let $\mu(t,\xi), t \in [0,T], \xi \in \R,$ be a random field
of the type
$$ \mu(t,\xi) = \sum_{i=1}^\infty e^i (\xi) W^i_t + e^0(\xi) t,
 \ t \in [0,T], 
\xi \in \R,$$
where $W^i, i \ge 1$
are independent continuous
 $(\shf_t)$-Brownian motions   on $(\Omega, \shf, P)$, 
 which are fixed from now on until the end of the paper.
For technical reasons we will sometimes set
$W^0_t \equiv t$.

We now consider  a random field $a$ and a deterministic function $\psi$ 
as follows.
\begin{Assumption} \label{ABorel}
 $a:[0,T]\times\R\times\O\to\R_+$ is a bounded progressively
measurable random field.
\end{Assumption}
\begin{Assumption} \label{ALipschitz}
 $\psi$ is monotone increasing, Lipschitz such that
 $\psi(0) = 0$.
\end{Assumption}

 Let $x_0 \in \shs'(\R)$.
We will consider the following two types of equations \eqref{e5.1} and  
\eqref{PME}. The first one 
is a (linear) stochastic Fokker-Planck equation, the second one
a stochastic porous media type equation, i.e.
\begin{align}\label{e5.1}
\begin{cases} 
\partial_t z(t,\xi) = \partial^2_{\xi\xi} ((a z)(t,\xi)) + z(t,\xi) 
\partial_t\mu (t, \xi),\\
z(0,\;\cdot\;) = x_0,
\end{cases}
\end{align}
and
\begin{equation}
\label{PME}
\left \{
\begin{array}{ccl}
\partial_t X(t,\xi)&=& \frac{1}{2} \partial_{\xi \xi}^2(\psi(X(t, \xi) ) +
 X(t,\xi) \partial_t \mu(t, \xi),
\\
X(0,\d \xi)& = & x_0. 
\end{array}
\right.
\end{equation}
They are both to be understood in the sense of  (Schwartz) distributions.
Their precise sense will be given  in Remark \ref{R5.2} a) 
and in Definition \ref{DSPDE}.
The stochastic multiplication above is of It\^o type.
In this paper we confine ourselves to the case of the underlying
space being $\R^1$.

Fokker-Planck equations have been investigated until now in
the deterministic framework, i.e. when $e^i = 0, i \ge 1$.
There is a huge literature about existence and uniqueness in this case,
 see e.g. \cite{BKRS} 
and references therein. More particularly, concerning uniqueness, in addition 
we draw the  attention 
to  Proposition 3.4 \cite{BRR1} 
and Theorem 3.1 of \cite{BCR2}.
As far as we know this is  the first  time that  a Fokker-Planck 
equation as \eqref{e5.1}  is considered in the literature, in particular for uniqueness, 
except for   the unpublished work by the same authors \cite{Bar-Roc-Rus-2200}.
 We point out that we can allow degenerate coefficients in the second order 
term.

Concerning porous media equations, both in the deterministic
and stochastic cases, there is a huge number of contributions,
especially in finite volume. As far as the infinite volume case is concerned,
in the deterministic situation a good framework   is the classical Benilan-Crandall approach 
of the seventies; in the stochastic case some recent significant contributions
have been made, see \cite{Ren, wang2008, BRR3} and in particular \cite{BarbuPorousBook}
and references therein.
 As mentioned, this paper draws 
however the attention on uniqueness for equations in the sense of distributions,
within a large solutions class.  For instance, in
the deterministic case a typical result in that sense is the paper \cite{BrC79}
of Brezis and Crandall, which establishes uniqueness in the sense of distributions 
in the class $(L^1 \cap L^\infty)([0,T] \times \R^d)$. 
Here  we consider the equation
\eqref{PME} in the sense of distributions and we investigate
uniqueness in the  class of progressively measurable random fields $X:\Omega \times [0,T] \times \R $ such that $ \int_{[0,T]\times\R} X^2(s,\xi)\d s\d \xi <\9$ a.s, see Definition \ref{DSPDE} and condition below \eqref{eFB1}.
 To the best of our knowledge, this constitutes a new result of uniqueness
 in the sense of distributions; for this we need only a.s.
 conditions in \eqref{eFB1} and not necessarily in the expectation 
as it is mostly done in the standard literature.


The paper is organized as follows. After, this introduction
and Section \ref{SPrelim}  devoted to preliminaries,
in Section \ref{S5} 
 the uniqueness Theorem \ref{T51} for an SPDE of Fokker-Planck type is
 formulated and proved. This, in turn, is   an important ingredient for the 
probabilistic representation of a solution to a
stochastic porous media type equation, see
\cite{Bar-Roc-Rus-2200}.
In the final section \ref{SUniqPME}, using the same ideas
as in Section \ref{S5}  we prove a uniqueness result for 
\eqref{PME}, see Theorem \ref{TB1}.

\section{Preliminaries} 
\label{SPrelim}

\setcounter{equation}{0}


First we introduce some basic recurrent notations.
$\shm(\R)$ 
 denotes the space of
signed Borel measures with finite total variation.
We recall that  $\shs(\R)$ is the space of the Schwartz fast decreasing
test functions with its usual topology.  $\shs'(\R)$ is its dual, i.e. 
the space of Schwartz tempered distributions.
On $\shs'(\R)$, the map $(I-\Delta)^\frac{s}{2},  s \in \R,$ is well-defined,
via Fourier transform.
For $s \in \R$, $H^s(\R)$ denotes the classical
Sobolev space consisting of all functions  $f \in \shs'(\R)$ such that 
$(I-\Delta)^\frac{s}{2} f \in L^2(\R)$.
We introduce   the norm
 $$\Vert f \Vert_{H^s} := \Vert (I-\Delta)^\frac{s}{2} f \Vert_{L^2},$$
where $\Vert \cdot \Vert_{L^p}$ is the classical $L^p(\R)$-norm
for $ 1 \le p \le \infty$.
In the sequel, we will often simply denote $H^{-1}(\R)$, 
by  $H^{-1}$ and $L^2(\R)$  by $L^2$.
Furthermore, $W^{r,p}$ denotes the classical Sobolev space of order $r \in \N$
in $L^p(\R)$ for $1 \le p \le \infty$. 
\begin{defi} \label{DMultipl}
Given a function $e$ belonging to $L^1_{\rm loc}(\R) \cap  \shs'(\R)$, we say that it is an
{\bf $H^{-1}$-multiplier}, if the map
$ \varphi \mapsto \varphi  e$ 
is continuous from $\shs(\R)$ to $H^{-1}$ 
with respect to the $H^{-1}$-topology on both spaces.
  $\shc(e)$
denotes the norm of this operator and we will call it {\bf multiplier norm}.
We remark that $  \varphi e$ is always 
 a well-defined Schwartz tempered distribution, 
 whenever $\varphi$ is a fast decreasing test function.
\end{defi}
\begin{rem}  \label{LMultipl}
Let $e\, : \, \R \to \R$.  If  $e\in W^{1,\infty}$ 
(for instance  if $e \in W^{2,1}$), 
 then $e$ is a $H^{-1}(\R)$-multiplier.\\
Indeed, by duality arguments, to show this, 
it is enough to show the existence of
a constant $\shc(e)$ such that
\begin{equation}\label{Lmult1} 
\norm{eg}_{H^1} \leq \shc(e)\norm{g}_{H^1},\; \forall\; g \in \Sscr (\R).
\end{equation}
Now \eqref{Lmult1} follows easily by the derivation product rules
with for instance $$\shc(e)=\sqrt 2 \left(\norm{e}^2_\infty + 
\norm{e'}^2_\infty\right)^{\frac12}. $$
\end{rem}

Here we fix some conventions concerning measurability.
Any  topological space $E$ is naturally equipped with its
Borel $\sigma$-algebra $\shb(E)$. 
For instance
$\shb(\R)$ (resp. $\shb([0,T]$) denotes 
the Borel $\sigma$-algebra
of $\R$ (resp. $[0,T]$).

In the whole paper, the following assumption on $\mu$ will be in force.
\begin{Assumption} \label{AMultipl}
\begin{enumerate}
\item Each $e^i, i \ge 0,$ belongs to the Sobolev space
$W^{1,\infty}$.
\item
 $ \sum_{i=1}^\infty  \left( \Vert (e^i)'\Vert_\infty^2 
  +  \Vert e^i\Vert_\infty^2 \right) < \infty.$ 

\end{enumerate}

\end{Assumption}

With respect to  the random field $\mu$, we introduce a notation 
for the It\^o type  stochastic integral below.

Let  $ Z = (Z(s, \xi), s \in [0,T], \xi \in \R)$ be a random field
on $(\Omega, \shf, (\shf_t), P) $
such that $\int_0^T \left (\int_{\R}  \vert Z(s,\xi) \vert \d \xi \right)^2 
\d s < \infty$
a.s. and it is an $L^1(\R)$-valued, $(\shf_s)$-progressively measurable
 process.
 Then, provided, Assumption \ref{AMultipl} holds, the stochastic integral
\begin{equation} \label{DSI}
\int_{[0,t]\times \R} Z(s, \xi) \mu(\d s, \xi)  := \sum_{i=0}^\infty \int_0^t 
 \left(\int_\R  e^i(\xi) Z(s,  \xi) \d \xi \right) \d W^i_s, t \ge 0,
 \end{equation}
is well-defined. \\More generally, if  
$ s \mapsto Z(s,\cdot)$ is a measurable map 
$[0,T] \times \Omega \mapsto \shm(\R)$,
 such that $\int_0^T \Vert Z(s,\cdot) \Vert_{\rm var}^2 \d s < \infty$, 
then the stochastic integral 
\begin{equation} \label{DSIbis}
\int_{[0,t]\times \R} Z(s,\d \xi)  \mu(\d s, \xi)  := \sum_{i=0}^\infty \int_0^t 
 \left(\int_\R e^i(\xi)  Z(s,\d \xi) \right) \d W^i_s, t \ge 0,
 \end{equation}
is well-defined.

 \section{On the uniqueness of a Fokker-Planck type SPDE}
\label{S5}

\setcounter{equation}{0}

The theorem below plays the analogous role as Theorem 3.8 in \cite{BRR1}
or Theorem 3.1 in \cite{BCR2}.
We recall that our Fokker-Planck SPDE has
 possibly degenerate measurable coefficients.

\begin{theo} \label{T51}
We suppose that Assumptions \ref{ABorel} and \ref{AMultipl} hold.
Let $z^1, z^2$ be two measurable random fields belonging $\o$ a.s. to 
$C([0,T],\shs'(\R))$ such that $z^1, z^2: ]0,T]\times \O \to \shm (\R)$.
We  moreover suppose  the following.
\begin{enumerate}
\item[i)] $z^1-z^2 \in L^2 ([0,T]\times \R)$ a.s.
\item[ii)] $t \mapsto (z^1-z^2)(t,\cdot)$ is an
 $(\shf_t)$-progressively measurable
$\shs'(\R)$-valued process.
\item[iii)] $z^1,z^2$ are solutions to \eqref{e5.1}.
 such that
$\int_0^T \Vert z^i(s,\cdot)\Vert^2_{\var} \d s < \infty$ a.s. 
\end{enumerate}
Then $z^1 \equiv z^2$.
\end{theo}

\begin{rem} \label{R5.2}
\begin{enumerate}
 \item[a)] 
By a solution $z$ of equation \eqref{e5.1} we mean the following: 
for every $\varphi \in \shs(\R),  \forall t \in [0,T]$, 
\begin{align}\label{e5.2}
\int_\R \varphi(\xi) z(t,\d \xi) &= \left< x_0, \varphi \right>  + \int^t_0 \d s \int_\R a (s,\xi) \varphi''(\xi) z(s,\d \xi)\\
&+\int_{[0,t]\times \R}  \varphi(\xi)  z(s,\d \xi)  \mu(\d s, \xi)
 \quad \text{a.s.}
\end{align}
\item[b)]  Let $z = z^1 - z ^2$. 
Since  $z$  is $\o$ a.s. in 
$L^2([0,T];L^2(\R) \cap \shm(\R)) \; \subset L^2([0,T];H^{-1}(\R))$,
 $\int^t_0   z(s, \cdot)  \mu(\d s, \; \cdot\;)$ 
belongs  $\o$ a.s. to $C([0,T];H^{-1}(\R))$ and so also to
$ C([0,T];H^{-2}(\R))$ $\o$ a.s.
On the other hand  $\int^t_0 (az)''(s, \cdot)\d s$ can be seen as a Bochner
 integral in $H^{-2}(\R)$.
In particular any solutions $z^1,z^2$ to \eqref{e5.1} are such that 
$z = z^1 - z^2$
admits a modification whose paths belong (a.s.)
to  $C([0,T];H^{-2}(\R)) \cap L^2([0,T];L^2(\R))$. Since $z^i, i =1,2$, are continuous with values in $\shs'(\R)$,
 their difference is indistinguishable  with the mentioned modification.
\\
Consequently for $\o$ a.s. $z(t,\cdot) \in C([0,T];H^{-2}(\R))$ and
 outside a $P$-null set $N_0$,
we have (in $\shs'(\R)$ and  $H^{-2}(\R)$ )
\begin{equation}\label{e5.3}
z(t,\cdot) = \int^t_0 (az)''(s,\cdot)\d s+\int_0^t z(s,\cdot) \mu (\d s,\cdot). 
\end{equation}
\item[c)] By assumption i),
possibly enlarging the $P$-null set $N_0$ we get the following.
For $\o\notin N_0$,
for almost all $t\in ]0,T]$,
$ \( \int^t_0 (az) (s,\cdot) \d s\right)''\in H^{-1}(\R)$ and so 
$\int^t_0 (az) (s,\;\cdot\;) \d s\in H^1\; \d t$ a.e.
\end{enumerate}
\end{rem}

\begin{proof}  [Proof of Theorem \ref{T51}]
Let $z = z^1-z^2$.

We fix the null set $N_0$ and so $\o$ will always lie outside $N_0$ introduced
in  Remark \ref{R5.2} c).
Let $\phi$ be a mollifier with compact support and $\phi_\varepsilon = \frac{1}{\varepsilon} \phi (\frac{\cdot}{\varepsilon})$ be a  generalized 
sequence of mollifiers converging to the Dirac delta function. We set
$$
g_\varepsilon(t) = \norm{z_\varepsilon (t)}^2_{H^{-1}}
= \int_\R z_\varepsilon (t,\xi) ((I-\Delta)^{-1} z_\varepsilon )(t,\xi)  \d \xi, 
$$
where $z_\varepsilon(t,\xi) = \int_\R \phi_\varepsilon (\xi-y) z(t,\d y) $.
Since $t \mapsto z(t,\;\cdot\;)$ is continuous in $H^{-2}(\R)$, 
$t\mapsto z_\varepsilon (t,\;\cdot\;)$ is continuous in $L^2(\R)$ and so also in $H^{-1}(\R)$.
 We look at the equation fulfilled by $z_\varepsilon$. The identity \eqref{e5.3}  produces the  following equality in $L^2(\R)$ and so in $H^{-1}(\R)$:
\begin{align}\label{e5.4}
z_\varepsilon (t,\;\cdot\;) 
&= \int^t_0  \left\{ \left[ \( a(s,\;\cdot\;) z (s,\;\cdot\;) \)
\star \phi_\varepsilon \right]'' - \(a(s,\;\cdot\;)z(s,\;\cdot\;)\)\star 
\phi_\varepsilon\right\}\d s\\
&+ \int^t_0 (a(s,\;\cdot\;)z(s,\;\cdot\;))\star \phi_\varepsilon  \d s
+ \sum_{i=0}^\infty \int^t_0   (e^i z)(s,\;\cdot\;) \star \phi_\varepsilon \d W^i_s.
 \nonumber
\end{align}
We apply $(I-\Delta)^{-1}$ and we get
\begin{align} \label{e5.4bis}
(I-\Delta)^{-1} z_\varepsilon (t,\;\cdot\;) &= -\int^t_0  (a(s,\;\cdot\;)
 z(s,\;\cdot\;)) \star \phi_\varepsilon\d s \\
&+ \int^t_0  (I-\Delta)^{-1} \left[ \(a(s,\;\cdot\;) z (s,\;\cdot\;) \)\star \phi_\varepsilon \right] \d s\\
&+ \sum_{i=0}^\infty \int^t_0 (I-\Delta)^{-1} (e^i z)(s,\;\cdot\;) 
\star \phi_\varepsilon \d W^i_s.
\end{align}
We apply It\^o's formula to $g_\varepsilon$.
For a general introduction to  infinite dimensional Hilbert space 
valued stochastic calculus, see \cite{dpz}, \cite{prevot} or \cite{liu}. 
Taking into account, \eqref{e5.4}, \eqref{e5.4bis} and that $\left<f,g\right>_{H^{-1}}=\left<f,(I-\Delta)^{-1}g\right>_{L^2}$, we now obtain
\begin{align}\label{e5.5}
 g_\varepsilon (t) &= 2\int^t_0 \left<z_\varepsilon (s,\;\cdot\;),\;\d z_\varepsilon (s,\;\cdot\;)\right>_{H^{-1}}  \\
&+ \sum_{i=1}^\infty \int^t_0 \left< (e^i z (s,\;\cdot\;)) 
\star \phi_\varepsilon, \;  
(e^i z (s,\;\cdot\;)) 
 \star \phi_\varepsilon  \right>_{H^{-1}} \d s \\
&= -2\int^t_0 \< z_\varepsilon (s,\;\cdot\;),\(a(s,\;\cdot\;)z(s,\;\cdot\;)\) \star \phi_\varepsilon \>_{L^2} \d s  \\
&+ 2\int^t_0  \< z_\varepsilon (s,\;\cdot\;),\; (I-\Delta)^{-1} \( \( a(s,\;\cdot\;)z(s,\;\cdot\;)\)\star\phi_\varepsilon \)\>_{L^2}\d s  \\
&+ \sum_{i=1}^\infty \int^t_0 \left< (e^i z (s,\;\cdot\;)) \star \phi_\varepsilon,
 \;  (e^i z (s,\;\cdot\;)) 
 \star \phi_\varepsilon  \right>_{H^{-1}} \d s \\
&+ 2  \int^t_0  \< z_\varepsilon (s,\;\cdot\;), (e^0 z)(s,\cdot)
 \star \phi_\varepsilon \right>_{H^{-1}} \d s  + M^\varepsilon_t
\end{align}
 where
\begin{equation}\label{e5.6}
  M^\varepsilon_t = 2 \sum^\infty_{i=1} \int_0^t 
\left< z_\varepsilon (s,\;\cdot\;),\; (e^i z)(s,\;\cdot\;) \star \phi_\varepsilon
\right>_{H^{-1}} \d W^i_s.
\end{equation}
Below we will justify that \eqref{e5.6} is well-defined.
We summarize \eqref{e5.5} into
\begin{equation*}  
g_\varepsilon(t) = \wt g_\varepsilon(t)+M^\varepsilon_t, t \in [0,T].
\end{equation*}
We remark that 
\begin{align} \label{e5.6bis}
\sum^\infty_{i=1} \int_0^T \left (\< z
 (s,\;\cdot\;),\; (e^i z) (s,\;\cdot\;) \>_{H^{-1}}\right)^2  \d s 
&= \sum^\infty_{i=1} \int_0^T \left (\< (e^i z)(s,\;\cdot\;),\; (I - \Delta)^{-1} z (s,\;\cdot\;) \>_{L^2}\right)^2  \d s  \\
& \le \sum^\infty_{i=1} \int_0^T \Vert  (e^i z) (s,\;\cdot\;)\Vert_{L^2}^2 \Vert z(s,\cdot) \Vert_{H^{-2}}^2 \d s \\
&\le  \sum_{i=1}^\infty \Vert e^i\Vert_\infty^2
\sup_{s \in [0,T]} \Vert z(s,\cdot) \Vert_{H^{-2}}^2 
 \int_0^T \Vert z (s,\;\cdot\;)\Vert_{L^2}^2  \d s,  
\end{align}
because $z: [0,T] \rightarrow H^{-2}$ is a.s. continuous by Remark
 \ref{R5.2} b).

Consequently,
\begin{equation*}
M_t = \sum^\infty_{i=1} \int^t_0 \< z (s,\;\cdot\;),\; (e^i z)(s,\;\cdot\;)\>_{H^{-1}}
 \d W^i_s,
\end{equation*}
is a well-defined local martingale.

It is also not difficult to show that 
for $\varepsilon > 0$, 
 $$  \sum^\infty_{i=1}   \int^T_0\< z_\varepsilon
   (s,\;\cdot\;),\; (e^i z)(s,\;\cdot\;) \star \phi_\varepsilon \>_{H^{-1}}^2
 \d s < \infty, \ {\rm a.s.}$$
and so $M^\varepsilon$ defined in \eqref{e5.6} is also a well-defined local martingale.

By assumption we have of course $(\o \not \in N_0)$
\begin{equation}\label{e5.7}
\int_{[0,T]\times \R} (z_\varepsilon (s,\xi)-z(s,\xi))^2 \d s \d \xi 
\quad
\substack{\longrightarrow\\ \varepsilon \to 0} 0,
\end{equation}
\begin{equation}\label{e5.8}
\int_{[0,T]\times \R} ((az) \star \phi_\varepsilon -az)^2 (s,\xi) \d s \d \xi
 \quad \substack{\longrightarrow\\ \varepsilon \to 0} 0,
\end{equation}
\begin{equation}\label{e5.8bis}
\int_{[0,T]\times \R} ((e^i z)(s,\cdot) \star \phi_\varepsilon - (e^i z)(s,\cdot))^2 (\xi) \d s \d \xi \quad
\substack{\longrightarrow\\ \varepsilon \to 0} 0,
\end{equation}
for every $ i \ge 0$,
because $z,az,  e^i z \in L^2([0,T]\times \R),  i \ge 0$.
By usual estimates on convolutions, there is a universal constant $C$
such that 
\begin{equation} \label{e500}
\int_{[0,T]\times \R} ((e^i z)(s,\cdot) \star \phi_\varepsilon)^2(\xi) \d s \d \xi
 \le
 \int_{[0,T]\times \R} z^2(s,\xi) (e^i)^2(\xi) \d s \d \xi  \le
C \Vert e^i \Vert_\infty^2 \Vert z \Vert^2_{L^2([0,T] \times \R)}. 
\end{equation}
By Lebesgue dominated convergence theorem, using \eqref{e5.8bis}, it follows that
 (for $\o \notin N_0$), 
\begin{equation}\label{e5.8ter}
\sum^\infty_{i=0} \int^T_0 \Vert  (e^i z) (s,\;\cdot) \star \phi_\varepsilon  - 
e^i z(s,\;\cdot\;)  \Vert_{L^2}^2 \d s \rightarrow_{\varepsilon \rightarrow 0} 0.
\end{equation}
Using \eqref{e5.7} and \eqref{e5.8ter}, it is not difficult to show that
(for $\o \notin N_0$)
\begin{equation}\label{e5.9}
\sum^\infty_{i=0} \int^T_0 \( 
\< z_\varepsilon (s,\;\cdot\;), (e^i z) (s,\cdot) \star \phi_\varepsilon \>_{H^{-1}}
 -\< z(s,\;\cdot\;), (e^i z) (s,\;\cdot\;)\>_{H^{-1}} \)^2 \d s
\end{equation}
converges to zero.
Now   (for $\o \notin N_0$), 
\begin{align}\label{e5.10}
& \sum^\infty_{i=0}  \int^T_0 \left \vert \norm{(e^i z)(s,\;\cdot\;) \star \phi_\varepsilon}^2_{H^{-1}}-\norm{(e^i z)(s,\;\cdot\;)}^2_{H^{-1}}\right \vert \d s 
\\
&\le \sqrt{2} \sum^\infty_{i=0} 
\sqrt{\int^T_0 \( \norm{(e^i z)(s,\;\cdot\;)  \star \phi_\varepsilon -
 (e^i z)(s,\;\cdot\;)}^2_{H^{-1}}\)\d s
 \int^T_0 \(  \norm{(e^i z)(s,\;\cdot\;)  \star \phi_\varepsilon}_{H^{-1}}^2 + 
 \norm{(e^i z)(s,\;\cdot\;)}^2_{H^{-1}}\)\d s } \\
& \le \sqrt{2} \sqrt{\sum^\infty_{i=0} 
\int^T_0 \( \norm{(e^i z)(s,\;\cdot\;) \star \phi_\varepsilon -
 (e^i z)(s,\;\cdot\;)}^2_{H^{-1}}\)\d s} 
\sqrt{\sum^\infty_{i=0} 
 \int^T_0 \( \norm{(e^i z)(s,\;\cdot\;)  \star \phi_\varepsilon}_{H^{-1}}^2 + 
 \norm{(e^i z)(s,\;\cdot\;)}^2_{H^{-1}}\)\d s }
\end{align}
converges to zero, because of \eqref{e500} and \eqref{e5.8ter} and
Assumption \ref{AMultipl}.

Taking into account \eqref{e5.7}, \eqref{e5.8}, \eqref{e5.9} and 
\eqref{e5.10} we obtain (for $\o \notin N_0$), that 
$\lim\limits_{\varepsilon \to 0} \wt g_\varepsilon (t) = \wt g (t),\; t \in [0,T]$,
 where 
\begin{align}\label{e5.11}
\wt g (t) = & -2 \int^t_0 \<z(s,\;\cdot\;), a(s,\;\cdot\;) z(s,\;\cdot\;)\>_{L^2} \d s\\
& + 2 \int^t_0  \<z(s,\;\cdot\;), (I-\Delta)^{-1} (a(s,\;\cdot\;) z(s,\;\cdot\;))\>_{L^2} \d s\\
& + 2 \int^t_0  \<z(s,\;\cdot\;), (e^0 z)(s,\;\cdot\;)\>_{H^{-1}} \d s \\
& + \sum^\infty_{i=1} \int^t_0 \<(e^i z)(s,\;\cdot\;), (e^i z)(s,\;\cdot\;) 
 \>_{H^{-1}} \d s.
\end{align}
The convergence of the second term of the right-hand side of \eqref{e5.5} 
to the second term of the right-hand side of \eqref{e5.11} holds again 
due to \eqref{e5.7} and \eqref{e5.8}, cutting the difference in two pieces and using Cauchy-Schwarz's inequality. On the other hand the convergence of \eqref{e5.9} 
to zero implies that 
$M^\varepsilon \to M$  ucp, so that the ucp limit of $\wt g_\varepsilon (t)
+M^\varepsilon_t$ is equal to $\wt g(t)+M_t$. So, after a possible modification of the
 $P$-null set $N_0$, setting $g(t) := \norm{z(t,\;\cdot\;)}^2_{H^{-1}}$,
 for $\o\notin N_0$, we have
\begin{align}\label{e5.12}
g(t)& +2 \int^t_0 \< z(s,\;\cdot\;), a(s,\;\cdot\;)z(s,\;\cdot\;)\>_{L^2} \d s\\
&= 2 \int^t_0 \< (I-\Delta)^{-1} z(s,\;\cdot\;), a(s,\;\cdot\;)z(s,\;\cdot\;)\>_{L^2} \d s  \\
&+ 2 \int^t_0 \d s \< z(s,\;\cdot\;),  (e^0 z)(s,\;\cdot\;)  \>_{H^{-1}} \d s \\
&+ \sum^\infty_{i=1} \int^t_0 \< (e^i z)(s,\;\cdot\;),  (e^i z)(s,\;\cdot\;)
  \>_{H^{-1}} \d s  + M_t.
\end{align}
By the inequality
\begin{equation*}
2bc \leq   b^2 \norm{a}_\9  +  \frac{c^2}{\norm{a}_\9},
\end{equation*}
$b,c\in\R$, it follows that
\begin{align*}
2 \int^t_0 & < (I-\Delta)^{-1} z(s,\;\cdot\;),(az)(s,\;\cdot\;) >_{L^2} \d s\\
& \leq \norm{a}_\9 \int^t_0 \norm{(I-\Delta)^{-1} z(s,\;\cdot\;)}^2_{L^2} \d s\\
& + \frac{1}{\norm{a}_\9} \int^t_0 < (az)(s,\;\cdot\;),(az)(s,\;\cdot\;)>_{L^2} \d s\\
& \leq \norm{a}_\9 \int^t_0 \norm{z(s,\;\cdot\;)}^2_{H^{-2}} \d s\\
& + \frac{1}{\norm{a}_\9} \norm{a}_\9 \int^t_0 \< z(s,\;\cdot\;), az(s,\;\cdot\;)\>_{L^2} \d s.
\end{align*}
Since $\norm{\; \cdot \;}_{H^{-2}} \leq \norm{\; \cdot \;}_{H^{-1}}$, 
\eqref{e5.12} gives now (for $\o \notin N_0$),
\begin{align*}
g(t) & + \int^t_0 \< z(s,\;\cdot\;),(az)(s,\;\cdot\;) \>_{L^2} \d s\\
\leq M_t & + \sum_{i=1}^\infty \int^t_0 \< (e^i z)(s,\;\cdot\;), (e^i z)
(s,\;\cdot\;) \>_{H^{-1}} \d s  \\
& + 2 \int^t_0  \< z(s,\;\cdot\;), (e^0 z)(s,\;\cdot\;)\>_{H^{-1}} \d s 
 + \norm{a}_\9 \int^t_0 \norm{z(s,\;\cdot\;)}^2_{H^{-1}} \d s.
\end{align*}
Since $e^i,\; i \ge 0$ are $H^{-1}$-multipliers with norm $\shc(e^i)$,
  (for $\o \notin N_0$)
\begin{align}\label{e5.13}
g(t) & + \int^t_0 \< z(s,\;\cdot\;),(az)(s,\;\cdot\;) \>_{L^2} \d s\\
& \leq M_t + \shc \int^t_0 \norm{z(s,\;\cdot\;)}^2_{H^{-1}} \d s = 
M_t + \shc \int^t_0 g(s) \d s, \; \forall \; t\in [0,T],
\end{align}
where 
$$ \shc = \sum_{i=1}^\infty \shc(e^i)^2 + 2 \shc(e^0) + \Vert a \Vert_\infty.$$
We proceed now via localization which is possible because 
$ t \mapsto  \int^t_0 \norm{z(s,\;\cdot\;)}^2 \d s$ and $ t \mapsto   
 \Vert z(t,\cdot)\Vert_{H^{-2}}$
are continuous $P$ a.s.
 Let $(\varsigma^\ell)$ be the sequence of stopping times
\begin{equation} \label{TauLocalization}
\varsigma^\ell:= \inf \{ t\in [0,T] | \int^t_0 \d s \norm{z(s,\;\cdot\;)}^2_{L^2} \ge \ell, \Vert z(t, \cdot) \Vert^2_{H^{-2}} \geq \ell \}.
\end{equation}
If $\{ \; \}=\emptyset$ we convene that $\varsigma^\ell = +\9$. Clearly, the stopped processes $M^{\varsigma^\ell}$ are (square integrable) martingales starting at zero. We evaluate \eqref{e5.13} at $t\wedge\varsigma^\ell$.
Taking  expectation we get
$$
E(g(t\wedge\varsigma^\ell))  \leq \underbrace{E(M_{\varsigma^\ell \wedge t})}_{=0} +
\shc E\(\int_0^{t\wedge\varsigma^\ell} g(s) \d s \)
 \leq \shc \int_0^t \d s E(g(s\wedge\varsigma^\ell)).
$$
By Gronwall's lemma it follows that
$
 E(g(t\wedge\varsigma^\ell))=0\quad \forall\; \ell \in \N^\star.
$
Since $g$ is a.s. continuous and $\lim_{\ell \to \infty} t\wedge\varsigma^\ell = T$
 a.s., for every $t\in[0,T]$, by Fatou's lemma we get
$$
E(g(t))  = E\(\liminf_{\ell\to \9} g(t\wedge\varsigma^\ell) \) \leq \liminf_{\ell\to \9} E\( g(t\wedge\varsigma^\ell) \)=0,
$$
and the result follows.
\end{proof}

\section{Uniqueness for the porous media equation 
with  noise}

\label{SUniqPME}

We first discuss  first in which sense the SPDE \eqref{PME}
has to be understood.

 \begin{defi} \label{DSPDE}
 A  random field $X = (X(t, \xi, \omega), 
t \in [0,T], \xi 
 \in \R, \omega \in \Omega) $ is said to be a solution to \eqref{PME} if
 $P$ a.s. 
we have the following.
\begin{itemize}
\item $X \in C([0,T]; \shs'(\R)) \cap  L^2([0,T]; L^1_{\rm loc} (\R))$.
\item $X$ is an $\shs'(\R)$
-valued $(\shf_t)$-progressively measurable process.
\item 
For any test function $\varphi \in \shs(\R)$ with compact support,
 $ t \in ]0,T]$, 
 we have 
\begin{eqnarray} \label{EDist}
\int_\R X(t,\xi) \varphi(\xi) \d\xi &=& \int_\R  \varphi(\xi) x_0(\d\xi) +
\halb \int_0^t \d s \int_\R 
\psi(X(s,\xi,\cdot)) \varphi''(\xi) d\xi
 \nonumber\\
& & \\
&+& \int_{[0,t] \times \R} X(s,\xi) \varphi(\xi) \mu(\d s,\xi) \d\xi \ {\rm a.s.} 
\nonumber
\end{eqnarray}
\end{itemize}
\end{defi}

We can state now the uniqueness
theorem for the stochastic porous media equation.

\begin{theo} \label{TB1} 
Suppose that Assumptions \ref{ALipschitz} and \ref{AMultipl} hold.
Then equation \eqref{PME} admits at most one
 solution among the random fields $X:]0,T]\times\R\times\O\to\R$ such that 
\begin{equation}\label{eFB1}
\int_{[0,T]\times\R} X^2(s,\xi)\d s\d \xi <\9 \quad \text{a.s.}
\end{equation}
\end{theo}

\begin{rem}\label{remB2} Let $X$ be a solution of \eqref{PME} verifying 
\eqref{eFB1}. 

\begin{enumerate}[i)]
\item There is  a $P$-null set $N_0$, so that for $\o \not \in N_0,\;
 X(t,\;\cdot\;)\in L^2(\R)$ for almost all $t\in[0,T]$.
\item Condition \eqref{eFB1} also implies that
$\int^T_0 \norm{X(s,\;\cdot\;)}^2_{H^{-1}} \d s <\9 \quad \text{a.s.}$
\item Since $\psi$ is Lipschitz and $\psi(0) = 0$, \eqref{eFB1} implies that
$\int_0^T\Vert \psi(X(r,\cdot))\Vert_{L^2}^2 \d r < \infty$ a.s. So,  
$\int^t_0 \d s \psi (X(s,\cdot))$ is a Bochner integral with values in 
$L^2(\R)$.
\item Consequently, 
 $t\mapsto \Delta \( \int\limits^t_0  \psi (X(s,\;\cdot\;)) \d s \)$
 is continuous from $[0,T]$ to  $H^{-2}$ and so also in
 $\shs'(\R)$;
since $e^i,  i \ge 0,$ are $H^{-1}$-multipliers verifying Assumption
\ref{AMultipl}, 
by Kolmogorov's lemma
$t\mapsto \int\limits^t_0  X(s,\cdot) \mu (\d s, \cdot) $ 
admits a version which
belongs to $C\left([0,T]; H^{-1}(\R) \right)$.
Since $x_0 \in \shs'(\R)$  and $X \in C([0,T]; \shs'(\R))$ a.s.,
 it follows that
for $\o$ not belonging to a null set, we have
\begin{equation}\label{eFB3}
X(t,\;\cdot\;)=x_0+\Delta \(\int\limits^t_0 \psi (X(s,\;\cdot\;))\d s \right)
 + \int\limits^t_0   X(s,\cdot)  \mu ( \d s , \cdot), \quad t\in [0,T],
\end{equation}
as an identity in $\shs'(\R)$.
\item If $x_0 \in H^{-1}$, then $X\in C\([0,T];H^{-2}\)$, for $\o \not \in N_0$,
$N_0$ a $P$-null set.
\item If $x_0 \in H^{-s}$ for some $s \ge 2$, then $X\in C\([0,T];H^{-s}\)$, 
for $\o \not \in N_0$,
$N_0$ a $P$-null set.
\item We  consider a sequence of mollifiers $(\phi_\varepsilon)$ converging to the Dirac measure. Then $X^\varepsilon
 (t,\;\cdot\;) = X(t,\cdot) \star \phi_\varepsilon$ belongs a.s. 
to $C\([0,T];L^2(\R)\)$.
\end{enumerate}
\end{rem}
\begin{rem} \label{remB2ter}
Since $\psi$ is Lipschitz, there is $\alpha>0$ such that
\begin{equation*}
\(\psi(r)-\psi(\bar r)\) (r-\bar r)\geq \alpha \( \psi (r)-\psi(\bar r)\)^2.
\end{equation*}
\end{rem}
\begin{rem}\label{remB2bis} 
\begin{enumerate}
\item
We note that condition 2. in Assumption \ref{AMultipl}  is
 more general than (3.1) of \cite{BRR3}, which can be reformulated  here as
 follows. 
\begin{Assumption} \label{E3.1BRR3}
\begin{enumerate}
\item  $e^i \in W^{1,\infty}$ for every $i \ge 0$.
\item 
$e^i, i \ge 0,  \ {\rm belong \ to} \ H^1.$
\item  
$(e^i)$  is an orthonormal system of $H^{-1}$ and \\
$\sum_{i=1}^\infty  \left(\Vert (e^i)'\Vert_\infty^2 
  +  \Vert e^i\Vert_\infty^2  + \Vert e^i \Vert_{H^{-1}}^2 \right)  < \infty.$
\end{enumerate}
\end{Assumption}
\item  An easy adaptation of Theorem 3.4 of \cite{BRR3}, in order to take into account $e^0$, 
constitutes an existence result for
\eqref{PME}: it says the following. \\
 Besides  Assumptions \ref{E3.1BRR3} and \ref{ALipschitz}, let us 
suppose moreover that
 $x_0 \in L^2$ or $\psi$   is non-degenerate
(i.e. $ \frac{\psi(x)}{x} \ge 0, \ \forall x \neq 0$).
 Then, 
there is a random field $X$
such that 
\begin{equation}\label{eFB1bis}
E \left(\int_{[0,T]\times\R} X^2(s,\xi)\d s \d \xi\right) <\9,
\end{equation}
with
 $t \mapsto  \int_0^t \psi(X(s,\cdot)) \d s \in  C([0,T]; H^1(\R))$ a.s.
\item So, under the assumptions of item 2. above, the solution $X$
is unique among those fulfilling \eqref{eFB1}.
\end{enumerate}
\end{rem}

\begin{proof}[Proof]
Let $(\phi_\varepsilon, \varepsilon > 0)$ be a sequence of mollifiers as 
in Remark \ref{remB2} 
vii).
Let $X^1,X^2$ be two solutions of \eqref{PME}.
For $i = 1,2$,  we set 
$(X^i)^\varepsilon(t,\cdot) = X^i(t,\cdot) \star \phi_\varepsilon.$
We set $X=X^1 - X^2$ and   $X^\varepsilon=(X^1)^\varepsilon-(X^2)^\varepsilon$
 which a.s. belongs to $C\([0,T];\;L^2(\R)\) \subset C\([0,T];\;H^{-1}\)$.
We set
$$ 
g_\varepsilon(t) :=  \norm{X^\varepsilon(t,\;\cdot\;)}^2_{H^{-1}}
 = \int_\R \((I-\Delta)^{-1} X^\varepsilon (t,\;\cdot\;) \) (\xi) X^\varepsilon (t,\xi) \d \xi.
$$ 
It\^o's formula gives
\begin{equation} \label{eFB4}
g_\varepsilon(t) =  2\int^t_0 < X^\varepsilon (s,\;\cdot\;), X^\varepsilon (\d s,\;\cdot\;) >_{H^{-1}}  + \sum^\9_{i=1} \int^t_0 \norm{(e^i X)(s,\cdot) 
\star \phi_\varepsilon}^2_{H^{-1}} \d s.
\end{equation}
On the other hand we have
\begin{equation}\label{eFB5}
X^\varepsilon (t,\;\cdot\;) =  \int^t_0  \Delta \left[ \left\{\psi
 (X^1(s,\;\cdot\;))-\psi (X^2(s,\;\cdot\;))\right\}\star \phi_\varepsilon \right] \d s 
 + \int^t_0  \phi_\varepsilon \star  (X \mu(\d s,\;\cdot\;)),
\end{equation}
where the notation of the latter integral is self-explanatory.
 So
\begin{align}\label{eFB6}
(I-\Delta)^{-1} X^\varepsilon (t,\;\cdot\;) & = -\int^t_0 \(\psi(X^1(s,\;\cdot\;))
-\psi(X^2(s,\;\cdot\;))\)\star \phi_\varepsilon \d s  \\
& + \int^t_0 (I-\Delta)^{-1}\(\psi(X^1(s,\;\cdot\;))-\psi(X^2(s,\;\cdot\;))\)\star \phi_\varepsilon \d s   \\
& + \sum^\9_{i=0} \int^t_0  \left[(I-\Delta)^{-1}(e^i X(s,\;\cdot\;))\right]\star \phi_\varepsilon  \d W^i_s.
\end{align}
We define 
\begin{equation*}
M_t = \sum^\9_{i=1} \int^t_0 < (I-\Delta)^{-1} X (s,\;\cdot\;),
 e^i X(s,\;\cdot\;) >_{L^2} \d W^i_s.
\end{equation*}
We observe that $M$ is well-defined and it is
a local martingale. Indeed, by Remark \ref{remB2} v), 
$X \in C([0,T];H^{-2})$. So by   similar arguments as in \eqref{e5.6bis},
\begin{align} \label{EFB100}
 \sum^\9_{i=1} \int^t_0 < (I-\Delta)^{-1} X (s,\;\cdot\;), 
 e^i X(s,\;\cdot\;) >_{L^2}^2 \d s &\le 
\sup_{s \in [0,T]} \Vert X(s,\cdot) \Vert_{H^{-2}}^2 
\sum_{i = 1}^\infty  \Vert e^i\Vert_\infty^2  \nonumber \\
&\int_0^T \Vert X (s,\;\cdot\;)\Vert_{L^2}^2  \d s < \infty.
\end{align}
Using \eqref{eFB4}, \eqref{eFB5} and \eqref{eFB6} we get
\begin{align}\label{eFB7}
g_\varepsilon (t) = & \sum^\9_{i=1}\int^t_0  \norm{(e^i X)(s,\cdot)\star
 \phi_\varepsilon}^2_{H^{-1}} \d s  \\
& - 2\int^t_0 < X^\varepsilon (s,\;\cdot\;), \left[\psi (X^1(s,\;\cdot\;))-\psi 
(X^2(s,\;\cdot\;))\right]\star \phi_\varepsilon >_{L^2}   \d s \\
& + 2\int^t_0 < X^\varepsilon (s,\;\cdot\;), (I-\Delta)^{-1}\left[\psi
 (X^1(s,\;\cdot\;))-\psi (X^2(s,\;\cdot\;))\right]\star \phi_\varepsilon >_{L^2}  \d s \\
& + 2\int^t_0 < X^\varepsilon (s,\;\cdot\;), (I-\Delta)^{-1} \left[e^0 
X(s,\;\cdot\;)\right]\star \phi_\varepsilon >_{L^2}  \d s
 + M_t^\varepsilon,
\end{align}
where $M^\varepsilon$ is the local martingale defined by
\begin{equation*}
M^\varepsilon_t = \sum^\9_{i=1} \int^t_0 < X^\varepsilon (s,\;\cdot\;), 
(I-\Delta)^{-1} \(e^i X(s,\;\cdot\;)\) \star \phi_\varepsilon >_{L^2} \d W^i_s,
\end{equation*}
which is again well-defined by similar arguments
as in  the proof of \eqref{EFB100}.
Taking into account \eqref{eFB1} and the Lipschitz property for $\psi$, we can take the limit when $\varepsilon \to 0$ in \eqref{eFB7} and for $g(t):=\norm{X(t,\;\cdot\;)}^2_{H^{-1}}$, to obtain
\begin{align}\label{eFB8}
g(t) & + 2 \int^t_0  \< X(s,\;\cdot\;),\; \psi \(  X^1(s,\;\cdot\;)\)-\psi \(  X^2(s,\;\cdot\;)\)\>_{L^2} \d s  \\
= & \sum^\9_{i=1} \int^t_0  \norm{e^i X(s,\;\cdot\;)}^2_{H^{-1}}\d s  \\
& + 2 \int^t_0  \< (I-\Delta)^{-1}X(s,\;\cdot\;),\; \psi \(  X^1(s,\;\cdot\;)\)-\psi \(  X^2(s,\;\cdot\;)\)\>_{L^2} \d s \\
& + 2 \int^t_0  \< X(s,\;\cdot\;), e^0 X(s,\;\cdot\;)\>_{H^{-1}} \d s  + M_t.
\end{align}
The convergence $M^\varepsilon \to M$ when $\varepsilon \to 0$ is ucp, since
\begin{equation*}
\sum^\9_{i=1} \int^t_0 \big | \< X^\varepsilon (s,\;\cdot\;),\;
 [e^i X(s,\;\cdot\;)]\star \phi_\varepsilon  \>_{H^{-1}} 
- \< X (s,\;\cdot\;),\;
 e^i X(s,\;\cdot\;)  \>_{H^{-1}}  
\big |  ^2   \d s   \substack{\longrightarrow\\
 \varepsilon \to 0} 0,
\end{equation*}
which follows by similar arguments as in the proof of \eqref{e5.9}, using 
\eqref{eFB1}.

\begin{equation*}
2ab \leq \frac{a^2}{\alpha}+b^2\alpha,
\end{equation*}
for $a,b\in \R$, $\alpha$ being the constant
 appearing at  Remark \ref{remB2ter}, the second term 
of the right-hand side of equality \eqref{eFB8} is bounded by
\begin{align*}
& \frac{1}{\alpha} \int^t_0  \norm{(I-\Delta)^{-1}X(s,\;\cdot\;)}^2_{L^2} \d s
  + \alpha \int^t_0 
\norm{\psi \(  X^1(s,\;\cdot\;)\)-\psi \(  X^2(s,\;\cdot\;)\)}_{L^2}^2 \d s\\
\leq & \frac{1}{\alpha} \int^t_0  \norm{X(s,\;\cdot\;)}^2_{H^{-1}} \d s + 
\int^t_0  \< \psi \(  X^1(s,\;\cdot\;)\)-
\psi \(  X^2(s,\;\cdot\;)\),\; X(s,\;\cdot\;)\>_{L^2} \d s.
\end{align*}
This together with \eqref{eFB8} gives 
\begin{align}\label{eFB9}
g(t) & +\int^t_0 \< X(s,\;\cdot\;),\;\psi \(  X^1(s,\;\cdot\;)\)-\psi \(  X^2(s,\;\cdot\;)\)\>_{L^2} \d s  \\
\leq & 2 \int^t_0  \< X(s,\;\cdot\;),e^0 X(s,\;\cdot\;)\>_{H^{-1}} \d s   +
  \frac{1}{\alpha} \int^t_0  \norm{X(s,\;\cdot\;)}^2_{H^{-1}} \d s    \\
& + \sum^\9_{i=1}\int^t_0  \norm{(e^i X)(s, \cdot)}^2_{H^{-1}} \d s  +M_t, 
\ t \in [0,T] \pas  \\
\end{align}
Since $e^i,\; i\in \N$, are $H^{-1}$-multipliers and taking into 
account Assumption \ref{AMultipl}, we get
\begin{equation}\label{eFB10}
g(t)\leq M_t + (2 \shc_0 + \sum_{i = 1}^\infty \shc(e^i)^2 + \frac{1}{\alpha}) \int^t_0  g(s) \d s.
\end{equation}
The proof is then completed by localization as in \eqref{TauLocalization} at the end of Section \ref{S5}.
\end{proof}

\bigskip

{\bf ACKNOWLEDGMENTS} 
\noindent

Financial support through the SFB 701 at Bielefeld University and
NSF-Grant 0606615
is gratefully acknowledged.
The second named author  benefited partially from the support of the 
``FMJH Program Gaspard Monge in optimization and operation research'' 
(Project 2014-1607H). The authors are grateful to  Viorel Barbu
for stimulating discussions. The authors are grateful to the Referee whose comments have stimulated them
to drastically improve the first version of the paper.

\addcontentsline{toc}{section}{References}
\bibliographystyle{plain}
\bibliography{BRR_Bibliography}

\end{document}

%% file: preamble.tex
\usepackage[T1]{fontenc}
\usepackage{ae}
\usepackage{amssymb,amsmath,amsthm,textcomp,amsxtra}
\usepackage[mathcal]{euscript}
\usepackage{graphicx,pst-plot,mparhack}
\usepackage[abs]{overpic}
\usepackage[matrix,arrow]{xy}
  \xymatrixrowsep{0.3pc}
  \xymatrixcolsep{0.3pc}
\usepackage{enumerate,makeidx}

\renewcommand*{\geq}{\geqslant}
\renewcommand*{\leq}{\leqslant}

\makeindex

\numberwithin{equation}{section}

\newcommand*{\Sscr}{\mathcal S}

\newcommand*{\N}{\mathbb{N}}
\newcommand*{\R}{\mathbb{R}}

\renewcommand*{\d}{\mathrm{d}}

\DeclareMathOperator{\var}{var}